\documentclass[11pt,a4paper]{article}

\usepackage{amsmath}
\usepackage{amsfonts}
\usepackage{amsthm}
\usepackage{graphicx}
\usepackage{eucal}
\usepackage{amscd}
\usepackage[all,2cell]{xy}
\UseAllTwocells
\usepackage{amssymb}
\usepackage{mathrsfs}

\title{Characteristic subobjects in semi-abelian categories}
\author{Alan S. Cigoli and Andrea Montoli}

\newtheorem{teo}{Theorem}[section]
\newtheorem{cor}[teo]{Corollary}
\newtheorem{lemma}[teo]{Lemma}
\newtheorem{defi}[teo]{Definition}
\newtheorem{prop}[teo]{Proposition}
\theoremstyle{definition}\newtheorem{example}[teo]{Example}
\newtheorem{remark}[teo]{Remark}

\DeclareMathOperator{\Ker}{Ker}
\DeclareMathOperator{\Coker}{Coker}
\DeclareMathOperator{\coker}{coker}
\DeclareMathOperator{\Act}{Act} \DeclareMathOperator{\Aut}{Aut}
\DeclareMathOperator{\Der}{Der}

\newcommand{\C}{\ensuremath{\mathcal{C}}}
\newcommand{\Z}{\ensuremath{\mathbb{Z}}}
\newcommand{\car}{\ensuremath{\underset{\mathrm{char}}{<}}}

\newcommand{\Gp}{\ensuremath{\mathsf{Gp}}}

\newcommand{\Lie}{\ensuremath{\mathsf{Lie}}}
\newcommand{\NARng}{\ensuremath{\mathsf{NARng}}}
\newcommand{\Pt}{\ensuremath{\mathsf{Pt}}}

\newcommand{\Rng}{\ensuremath{\mathsf{Rng}}}

\newcommand{\LACC}{{\rm (LACC)}}
\newcommand{\NH}{{\rm (NH)}}

\def\skewpullback{%
 \ar@{-}[]+LD+<-6pt,-6pt>;[]+LDD+<-6pt,-15.5pt>%
 \ar@{-}[]+D+<-1pt,-6pt>;[]+LDD+<-6pt,-15.5pt>}

\newdir{|>}{{}*!/8pt/@{|}*!/3.5pt/:(1,-.2)@^{>}*!/3.5pt/:(1,+.2)@_{>}}
\newdir{ >}{{}*!/-10pt/@{>}}
\newdir{ |>}{!/10pt/{}*!/4.5pt/@{|}*:(1,-.2)@^{>}*:(1,+.2)@_{>}}

\begin{document}

\maketitle

\begin{abstract}
We extend to semi-abelian categories the notion of characteristic
subobject, which is widely used in group theory and in the theory
of Lie algebras. Moreover, we show that many of the classical
properties of characteristic subgroups of a group hold in the
general semi-abelian context, or in stronger ones.
\end{abstract}


\section{Introduction}

The notion of characteristic subgroup (which means a subgroup that
is invariant under all automorphisms of the bigger group) is
widely used in group theory. Examples of characteristic subgroups
are the centre and the derived subgroup of any group. The main
properties of characteristic subgroups are the following: if $H$
is a characteristic subgroup of $K$ and $K$ is a characteristic
subgroup of $G$, then $H$ is a characteristic subgroup of $G$;
moreover, if $H$ is characteristic in $K$ and $K$ is normal in
$G$, then $H$ is normal in $G$. These transitivity properties of
characteristic subgroups imply, for example, that the derived
series and the central series of a group are normal series, and
this fact is very useful in order to deal with solvable and
nilpotent groups.
\medskip

An analogous notion exists for Lie algebras (over a commutative
ring $R$): a characteristic ideal of a Lie algebra is a subalgebra
which is invariant under all derivations of the bigger one. The
two transitivity properties mentioned above hold also in this
context, and again this allows to easily describe solvable and
nilpotent Lie algebras.
\medskip

The strong parallelism between these two contexts is explained by
the fact that automorphisms represent group actions, as well as
derivations represent actions of Lie algebras in the following
sense. An action of a group $B$ on a group $G$ can be described
simply as a group homomorphism $B \to \Aut(G)$; in the same way,
an action of a Lie algebra $B$ on a Lie algebra $G$ is a
homomorphism of Lie algebras $B \to \Der(G)$.
\medskip

The aim of this paper is to extend the definition and the main
properties of characteristic subobjects to a categorical context.
In order to do this, we will use the notion of \emph{internal
action} introduced in \cite{BJK}. In
\cite{Bourn-Janelidze:Semidirect} it is proved that, in
semi-abelian categories \cite{Janelidze-Marki-Tholen}, internal
actions are equivalent to split extensions, via a semidirect
product construction which generalises the classical one known for
groups. Examples of semi-abelian categories are groups, rings,
associative algebras, Lie algebras and, in general, any variety of
$\Omega$-groups.
\medskip

We define a characteristic subobject as a subobject $H$ of an
object $G$ which is invariant under all (internal) actions over
$G$. In the semi-abelian context, we can use the equivalence
between actions and split extensions mentioned above in order to
deduce properties of characteristic subobjects from properties of
the kernel functor which associates with any split epimorphism its
kernel.
\medskip

The paper is organized as follows: in Section \ref{sec:defi} we
give the definition of characteristic subobject and we prove some
properties that hold in any semi-abelian category, like the
transitivity properties mentioned at the beginning, or the fact
that the intersection of characteristic subobjects is
characteristic. Then we study properties that hold in stronger
contexts, such as:
\begin{itemize}
\item[-] the join of two characteristic subobjects is
characteristic (Section \ref{sec:joins});

\item[-] the commutator of two characteristic subobjects is
characteristic (Section \ref{sec:comm});

\item[-] the centraliser of a characteristic subobject is
characteristic (Section \ref{sec:centr}).
\end{itemize}
Some properties about actors of characteristic subobjects are
studied in Section \ref{sec:actors} in the context of action
representative categories \cite{BJK2,Borceux-Bourn-SEC} and
analogous results are proved in action accessible categories
\cite{BJ07}, replacing actors with suitable objects.


\section{Definition and basic properties} \label{sec:defi}

A characteristic subgroup of a group $G$ is classically defined as
a subgroup $H$ of $G$ which is invariant under all the
automorphisms of $G$. This means that any automorphism of $G$
restricts to an automorphism of $H$. Since the automorphism group
$\Aut(G)$ of a group $G$ classifies all the group actions on $G$,
a subgroup $H$ of a group $G$ is characteristic if and only if any
group action on $G$ restricts to an action on $H$.

In other algebraic contexts it is no longer true that
automorphisms classify actions, hence the notions of invariance
under automorphisms and under actions are different. As already
explained in the introduction, here we are interested in the
latter. In order to study it in a categorical setting, we are
going to use the notion of \emph{internal action}, introduced in
\cite{BJK}. We briefly recall the definition.
\medskip

Let $\mathcal{C}$ be a pointed category with finite limits and
finite coproducts. For any object $B$ in $\mathcal{C}$, we can
define the category $\Pt_B(\C)$ of \emph{points} over $B$, whose
objects are split epimorphisms $(A, p, s)$ with codomain $B$ and
whose arrows are commutative triangles of the following form, with
$p'f=p$ and $fs=s'$:
$$ \xymatrix{
    A \ar[rr]^f \ar@<.5ex>[dr]^p & & A' \ar@<.5ex>[dl]^{p'} \\
    & B \ar@<.5ex>[ul]^s \ar@<.5ex>[ur]^{s'}
} $$ We then get the two following functors:
$$ \Ker_B \colon \Pt_B(\C) \to \C \,, $$
given by $\Ker_B(A, p, s) = \Ker \, p$, and
$$ B + (-) \colon \C \to \Pt_B(\C) \,, $$
where $B + (X)$ is the point $\xymatrix{B + X
\ar@<.5ex>[r]^-{[1,0]} & B \ar@<.5ex>[l]^-{\iota_B}}$.

These functors give rise to an adjunction. The corresponding monad on
$\C$ is denoted by $B \flat (-) $. For any object $X \in \C$, we
have that $B \flat X$ is the kernel of the morphism $[1, 0] \colon
B + X \to B$. The algebras for this monad are called
\emph{internal actions}. The comparison functor associates with
every point $(A,p,s)$ an action $\xi$ as described in the
following diagram (where $X$ is the kernel of $p$ and $\xi$ is
induced by the universal property of $X$):
$$ \xymatrix{
    B\flat X \ar [r]^-{\ker[1,0]} \ar[d]_\xi & B+X \ar @<3pt> [r]^-{[1,0]} \ar [d]^{[s,k]}
        & B \ar @<3pt> [l]^-{\iota_{_B}} \ar @{=} [d] \\
    X \ar [r]^k & A \ar @<3pt> [r]^p & B \ar @<3pt> [l]^s
} $$ When \C\ is the category \Gp\ of groups, the elements of
$B\flat X$ are generated by formal sequences of type
$(b;x;b^{-1})$ with $b\in B$ and $x\in X$, and the internal action
$\xi$ is nothing but the realisation of these sequences in $X$,
that is $\xi(b;x;b^{-1})=bxb^{-1}$, or more properly
$\xi(b;x;b^{-1})=k^{-1}(s(b)k(x)s(b^{-1}))$ since the product is
actually computed in $A$.

Vice versa, given a group action $\xi$
of $B$ over $K$, we can always associate with it the semidirect
product $K\rtimes_\xi B$ and a point as in the following diagram
where the left hand side square is constructed as a pushout:
$$ \xymatrix{
    B\flat X \ar [r]^-{\ker[1,0]} \ar[d]_\xi & B+X \ar@<3pt>[r]^-{[1,0]} \ar[d]
        & B \ar @<3pt> [l]^-{\iota_{_B}} \ar @{=} [d] \\
    X \ar[r]_-{i_X} & X \rtimes_\xi B \ar@<3pt>[r]^-{p_B} & B \ar@<3pt>[l]^-{i_B}
} $$ We can repeat the same construction in the categorical
context mentioned above. However, in general, the bottom row is
not always a split short exact sequence. This is the case when the
comparison functor is an equivalence, as, for example, in any
semi-abelian category \cite{Janelidze-Marki-Tholen,Borceux-Bourn},
as shown in \cite{Bourn-Janelidze:Semidirect}, where the
categorical notion of semi-direct product is introduced.
\medskip

We are now ready to give the following definition:

\begin{defi} \label{defi:char.sub}
Let \C\ be a pointed category with finite limits and finite
coproducts. Let $G$ be an object in \C\ and $h \colon H
\rightarrowtail G$ a subobject. We say that $H$ is
\emph{characteristic} in $G$, and we write $H \car G$, if, for all
pairs $(B,\xi)$, with $B$ an object of \C\ and $\xi$ an internal
action of $B$ on $G$, the action $\xi$ restricts to the subobject
$H$. In other words, there exists a (unique) action
$\overline{\xi}$ of $B$ on $H$ which makes the following diagram
commute:
$$ \xymatrix{
    B \flat H \ar[d]_{\overline{\xi}} \ar[r]^{1 \flat h} & B \flat G \ar[d]^{\xi} \\
    H \ar@{ >->}[r]_h & G
} $$
\end{defi}

When \C\ is a semi-abelian category, the above mentioned
equivalence between actions and points allows us to reformulate
the definition of characteristic subobject.

\begin{prop}
Let \C\ be a semi-abelian category. A subobject $h \colon H
\rightarrowtail G$ is characteristic in $G$ if and only if for
every split extension of kernel $G$
$$ \xymatrix{G \ar@{ |>->}[r] & X \ar@<.5ex>[r] & B \ar@<.5ex>[l]} $$
there exist a split extension $\xymatrix{H \ar@{ |>->}[r] & Y
\ar[r] & B}$ and a morphism of split extensions inducing $h$ on
kernels and $1_B$ on cokernels (it is necessarily a monomorphism
thanks to the split short five lemma):
$$ \xymatrix{
    H \ar@{ |>->}[r] \ar@{ >->}[d]_h & Y \ar@<.5ex>[r] \ar@{ >->}[d] & B \ar@<.5ex>[l] \ar@{=}[d] \\
    G \ar@{ |>->}[r] & X \ar@<.5ex>[r] & B \ar@<.5ex>[l]
} $$
\end{prop}

As we will see afterwards, this reformulation makes the notion of
characteristic subobject much easier to handle. Moreover, the
translation in terms of points reveals that, when actions are
equivalent to points, many properties of characteristic subobjects
are strictly related with the properties of the fibration of
points (see \cite{Borceux-Bourn}) or, to be more precise, of the
kernel functors:
$$ \Ker_B \colon \Pt_B(\C) \rightarrow \C $$
For these reasons, in our investigation, we will focus on contexts
which are at least semi-abelian, possibly with additional
requirements. The behaviour of characteristic subobjects in weaker
contexts is material for future work.

\begin{prop} \label{prop:char.lifts}
If $H$ is a characteristic subobject of $K$, and $K$ is a
characteristic subobject of $G$, then $H$ is characteristic in
$G$.
\end{prop}

\begin{proof}
The result is a straightforward consequence of Definition
\ref{defi:char.sub}.
\end{proof}

\begin{prop} \label{prop:normality.lifts}
If $H$ is a characteristic subobject of $K$, and $K$ is a normal
subobject of $G$, then $H$ is normal in $G$.
\end{prop}

\begin{proof}
It suffices to observe that, in the semi-abelian context, normal
subobjects are exactly those closed under the conjugation action
(i.e.\ clots, see for example \cite{MaMe10-2}). Indeed, the
conjugation action of $G$ on itself restricts to $K$ by normality,
and then to $H$, since $H \car K$, thus proving that $H
\triangleleft G$.
\end{proof}

\begin{cor} \label{cor:char=>norm}
If $H$ is a characteristic subobject of $G$, then $H$ is normal in
$G$.
\end{cor}

\begin{prop} \label{prop:meet.char}
If $\{H_i\}_{i\in I}$ is a finite family of
characteristic subobjects of $G$, then the intersection
$\bigwedge_{i\in I} H_i$ is characteristic in $G$.
\end{prop}

\begin{proof}
If $h_i \colon H_i \rightarrowtail G$ is a characteristic
subobject, then, for every action $\xi \colon B \flat G \to G$,
there is a morphism in $\Pt_B(\C)$:
$$ \xymatrix{
    H_i \ar@{ |>->}[r] \ar@{ |>->}[d]_{h_i} & Y_i \ar@<.5ex>[r]^{p_i} \ar@{ >->}[d]
            & B \ar@<.5ex>[l]^{s_i} \ar@{=}[d] \\
    G \ar@{ |>->}[r] & G \rtimes_\xi B \ar@<.5ex>[r] & B \ar@<.5ex>[l]
} $$ Since the kernel functor $\Ker_B \colon \Pt_B(\C) \rightarrow
\C$ has a left adjoint, it preserves intersections, so the object
$\bigwedge_{i\in I} H_i$ in \C\ is the kernel of the intersection
$\bigwedge_{i\in I} (Y_i,p_i,s_i)$ in $\Pt_B(\C)$.
\end{proof}

When the category \C\ is not only semi-abelian, but also strongly
protomodular \cite{Bourn-strongly-protomod}, internal actions
behave well with respect to quotients. More precisely, in
\cite{MeQoa} the following result is proved.

\begin{prop} \label{prop:str.prot}
A semi-abelian category is strongly semi-abelian (i.e.\
semi-abelian and strongly protomodular) if and only if the
following property holds:
\begin{itemize}
    \item for every normal subobject $H \triangleleft G$ and every action $\xi \colon B \flat G \to G$, if $\xi$ restricts
    to $H$, then $\xi$ also induces a (unique) action $\widetilde{\xi}$ on the quotient $G/H$:
$$ \xymatrix{
    B \flat H \ar[r] \ar[d]_{\overline{\xi}} & B \flat G \ar[r] \ar[d]^{\xi}
            & B \flat (G/H) \ar[d]^{\widetilde{\xi}} \\
    H \ar@{ |>->}[r]_-h & G \ar@{-|>}[r]_-q & G/H
} $$
\end{itemize}
\end{prop}

In terms of split extensions, this means that if a kernel $h$ is
the restriction of some $\phi$ in $\Pt_B(\C)$, then $q=\coker(h)$
is the restriction of $\gamma=\coker(\phi)$ in $\Pt_B(\C)$:
\begin{equation} \label{eq:cok.pt}
\begin{aligned}
\xymatrix{
    H \ar@{ |>->}[r] \ar@{ |>->}[d]_h & Y \ar@<.5ex>[r] \ar@{ >->}[d]^\phi & B \ar@<.5ex>[l] \ar@{=}[d] \\
    G \ar@{ |>->}[r] \ar@{-|>}[d]_q & X \ar@<.5ex>[r] \ar@{-|>}[d]^\gamma & B \ar@<.5ex>[l] \ar@{=}[d] \\
    G/H \ar@{ |>->}[r] & Z \ar@<.5ex>[r] & B \ar@<.5ex>[l]
}
\end{aligned}
\end{equation}

In fact, it turns out that, for the special class of
characteristic subobjects, strong protomodularity is not needed in
order to transfer actions to the quotient.

\begin{prop} \label{prop:act.quot}
If $H$ is a characteristic subobject of $G$, then every action on
$G$ induces an action on the quotient $G/H$, as in the diagram of
Proposition \ref{prop:str.prot}.
\end{prop}

\begin{proof}
By Proposition \ref{prop:normality.lifts}, for every action $\xi
\colon B \flat G \to G$, $H$ is a normal subobject of $G
\rtimes_\xi B$. Then the arrow $Y \rightarrowtail G \rtimes_{\xi}
B$, induced by the restriction of $\xi$ to $H$, is a normal
monomorphism in $\Pt_B(\C)$, according to \cite[Proposition
6.2.1]{Borceux-Bourn}:
$$ \xymatrix{
    H \ar@{ |>->}[r] \ar@{ |>->}[d]^{h} & Y \ar@<.5ex>[r] \ar@{ >->}[d] & B \ar@<.5ex>[l] \ar@{=}[d] \\
    G \ar@{ |>->}[r] & G \rtimes_\xi B \ar@<.5ex>[r] & B \ar@<.5ex>[l]
} $$ By taking its cokernel, we get an exact sequence as in
diagram (\ref{eq:cok.pt}).
\end{proof}

\begin{prop} \label{prop:K.char}
If $H \leq K \leq G$, $H$ is characteristic in $G$ and $K/H$ is
characteristic in $G/H$, then $K$ is characteristic in $G$.
\end{prop}

\begin{proof}
Let us consider the following diagram
$$ \xymatrix{
    H \ar@{ |>->}[r] \ar@{=}[d] & K \ar@{-|>}[r] \ar@{ >->}[d]^k & K/H \ar@{ >->}[d]^{\widetilde{k}} \\
    H \ar@{ |>->}[r] & G \ar@{-|>}[r]_-q & G/H
} $$ The right hand side square is a pullback (this comes from the
fact that the category $\C$, being semi-abelian, is protomodular
\cite{Bourn-protomod}). By Proposition \ref{prop:act.quot} every
action of some $B$ on $G$ induces an action on $G/H$. By
assumption, the same action restricts to $K/H$. In terms of
points, we have a cospan in $\Pt_B(\C)$ whose restriction to the
kernels is the pair $(q,\widetilde{k})$. Now, since the kernel
functors preserve pullbacks, $K$ is the kernel of the pullback in
$\Pt_B(\C)$ of the same cospan, hence the action of $B$ on $G$
restricts to $K$.
\end{proof}

\begin{prop} \label{prop:R.char}
If $H$ is characteristic in $G$, then its corresponding
equivalence relation $R$ on $G$ is closed under actions on $G$,
i.e.\ there exists an action $R(\xi)$ of $B$ on $R$ which makes
the following diagram commute:
$$ \xymatrix{
    B \flat R \ar[d]_{R(\xi)} \ar@<.5ex>[r]^{1 \flat r_1} \ar@<-.5ex>[r]_{1 \flat r_2} & B \flat G \ar[d]^{\xi} \\
    R \ar@<.5ex>[r]^{r_1} \ar@<-.5ex>[r]_{r_2} & G
} $$
\end{prop}

\begin{proof}
By Proposition \ref{prop:act.quot} every action of some $B$ on $G$
induces an action on $G/H$. Now, since kernel functors preserve
pullbacks, $R$ is the kernel of the kernel pair in $\Pt_B(\C)$ of
the morphism $\gamma$ of diagram (\ref{eq:cok.pt}):
$$ \xymatrix{
    R \ar@{ |>->}[r] \ar@<.5ex>[d]^{r_1} \ar@<-.5ex>[d]_{r_2}
        & R_\gamma \ar@<.5ex>[r] \ar@<.5ex>[d] \ar@<-.5ex>[d]   & B \ar@<.5ex>[l] \ar@{=}[d] \\
    G \ar@{ |>->}[r] \ar@{-|>}[d]_q & X \ar@<.5ex>[r] \ar@{-|>}[d]^\gamma & B \ar@<.5ex>[l] \ar@{=}[d] \\
    G/H \ar@{ |>->}[r] & Z \ar@<.5ex>[r] & B \ar@<.5ex>[l]
} $$
\end{proof}

We can make explicit the previous proposition in the category \Gp.
It says that for all pairs $(x,y) \in R$ and for all $b \in B$,
the pair $(^bx,\, ^by) \in R$.

More in general, whenever $B$ acts on $G$, there is an induced
action on $G \times G$ (simply computing the product in
$\Pt_B(\C)$), and the inclusion $R \rightarrowtail G \times G$ is
compatible with the corresponding actions. However, this does not
mean that $R$ is a characteristic subobject of $G \times G$.


\section{Joins} \label{sec:joins}

While the outcomes listed in Section \ref{sec:defi} hold in the
very general case of semi-abelian categories, the property that
finite joins of characteristic subobjects are characteristic
(which is true in the category of groups, for example) seems to
hold only in stronger contexts.

An additional requirement, for a semi-abelian category, which
turns out to be crucial in this sense, is to ask that kernel
functors preserve jointly strongly epimorphic pairs. This is
equivalent to the fact that, for all pairs
$((Y,p_1,s_1),(Z,p_2,s_2))$ of objects in $\Pt_B(\C)$, the
canonical arrow in \C:
$$ \Ker_B(Y,p_1,s_1)+\Ker_B(Z,p_2,s_2) \rightarrow \Ker_B((Y,p_1,s_1)+(Z,p_2,s_2)) $$
is a regular epimorphism.

\begin{lemma}[\cite{CMM}] \label{lemma:jse<=>joins}
Let \C\ be a semi-abelian category. For any object $B \in \C$ the
kernel functor $\Pt_B(\C) \to \C$ preserves jointly strongly
epimorphic pairs if and only if it preserves binary joins.
\end{lemma}

It is worth noting that the previous lemma does not say, in
particular, that, under the assumption, kernel functors preserve
coproducts. A counterexample to this fact is given in the proof of
Proposition 6.2 in \cite{Gray12} for the category of commutative
(not necessarily unitary) rings.

\begin{prop} \label{prop:join.char}
Let \C\ be a semi-abelian category where kernel functors preserve
jointly strongly epimorphic pairs. If $H$ and $K$ are
characteristic subobjects of $G$, then $H \vee K$ is
characteristic in $G$.
\end{prop}

\begin{proof}
Being $H$ and $K$ characteristic, for every action of $B$ on $G$,
the cospan $\xymatrix@1{H \ar@{ >->}[r]^h & G & K \ar@{
>->}[l]_k}$ is the restriction to kernels of a cospan in
$\Pt_B(\C)$. By Lemma \ref{lemma:jse<=>joins}, $H \vee K$ is the
kernel of a point over $B$.
\end{proof}

A context in which the property of preservation of binary joins by
the kernel functor holds is the one of \emph{locally algebraically
cartesian closed} categories \cite{Bourn-Gray}. A semi-abelian
category $\C$ is said locally algebraically cartesian closed (or
simply LACC) if, for any morphism $p \colon E \to B$ in $\C$, the
change of base functor
\[ p^* \colon \Pt_B(\C) \to \Pt_E(\C), \] defined by taking
pullbacks along $p$, has a right adjoint. Examples of this
situation are the categories $\Gp$ of groups and $\Lie$ of Lie
algebras over a fixed commutative ring $R$. In this context the
kernel functors (which are change of base functors with $E=0$),
having right adjoints, preserve all finite colimits, and hence the
canonical arrow
$$ \Ker_B(Y,p_1,s_1)+\Ker_B(Z,p_2,s_2) \rightarrow \Ker_B((Y,p_1,s_1)+(Z,p_2,s_2)) $$
mentioned above is an isomorphism.
\medskip

Another context in which preservation of binary joins holds is
given by \emph{categories of interest} \cite{Orzech}. We recall
that a category of interest is a category $\mathcal{C}$ whose
objects are groups with a set of operation $\Omega$ and with a set
of equalities $\mathbb{E}$, such that $\mathbb{E}$ includes the
group laws and the following conditions hold. If $\Omega_i$ is the
set of $i$-ary operations in $\Omega$, then:
\begin{itemize}
\item[(a)] $\Omega = \Omega_0 \cup \Omega_1 \cup \Omega_2$;

\item[(b)] the group operations (written additively: $0, -, +$,
even if the group is not necessarily abelian) are elements of
$\Omega_0$, $\Omega_1$ and $\Omega_2$ respectively. Let
$\Omega_2^{\prime} = \Omega_2 \backslash \{+\}$,
$\Omega_1^{\prime} = \Omega_1 \backslash \{-\}$ and assume that if
$* \in \Omega_2^{\prime}$, then $\Omega_2^{\prime}$ contains
$*^{\circ}$ defined by $x *^{\circ} y = y * x$. Assume further
that $\Omega_0 = \{0\}$;

\item[(c)] for any $* \in \Omega_2^{\prime}$, $\mathbb{E}$
includes the identity $x * (y + z) = x * y + x * z$;

\item[(d)] for any $\omega \in \Omega_1^{\prime}$ and $* \in
\Omega_2^{\prime}$, $\mathbb{E}$ includes the identities $\omega(x
+ y) = \omega(x) + \omega(y)$ and $\omega(x) * y = \omega(x * y)$;

\item[(e)] \textbf{Axiom 1} $\ $ $x_1 + (x_2 * x_3) = (x_2 * x_3)
+ x_1$ for any $* \in \Omega_2^{\prime}$;

\item[(f)] \textbf{Axiom 2} $\ $ for any ordered pair $(*,
\overline{*}) \in \Omega_2^{\prime} \times \Omega_2^{\prime}$
there is a word $W$ such that
\[ (x_1 * x_2) \overline{*} x_3 = W(x_1(x_2 x_3), x_1(x_3 x_2),
(x_2 x_3)x_1, (x_3 x_2)x_1, \]
\[ x_2(x_1 x_3), x_2(x_3 x_1), (x_1
x_3)x_2, (x_3 x_1)x_2), \] where each juxtaposition represents an
operation in $\Omega_2^{\prime}$.
\end{itemize}

Examples of categories of interest are groups, Lie algebras,
rings, associative algebras, Leibniz algebras, Poisson algebras
and many others. Also in this context the kernel functors preserve
binary joins, as follows from \cite{CGrayVdL2} and Lemma
\ref{lemma:jse<=>joins} herein.

Since it will be useful later, we give here a description of
internal actions in \emph{categories of interest} (called
\emph{derived actions} in \cite{Casas-Datuashvili-Ladra}). In a
\emph{category of interest} \C, an action of an object $B$ on an
object $X$ is a set of functions:
$$ f_* \colon B \times X \to X\,, $$
one for each operation $*$ in $\Omega_2$ (we will write $b\cdot x$
for $f_+(b,x)$ and $b*x$ for $f_*(b,x)$, with $*\in\Omega_2'$),
such that the one corresponding to the group operation $+$
satisfies the usual axioms for a group action, the others are
bilinear with respect to $+$ and moreover the following axioms are
satisfied (for all $b,b_i\in B$, $x,x_i\in X$ and $*, \overline{*}
\in \Omega'_2$):
\begin{enumerate}
    \item $b\cdot(x_1*x_2)=x_1*x_2$;
    \item $x_1+(b*x_2)=(b*x_2)+x_1$;
    \item $(b_1*b_2)\cdot x=x$;
    \item $b_1\cdot(b_2*x)=b_2*x$;
    \item $\begin{array}{lll}
			(b * x_1) \overline{*} x_2 & = & W(b(x_1 x_2), b(x_2 x_1),(x_1 x_2)b, (x_2 x_1)b, \\
			& & x_1(b x_2), x_1(x_2 b),(b x_2)x_1, (x_2b)x_1);
		\end{array}$
    \item $\begin{array}{lll}
			(x_1 * x_2) \overline{*} b & = & W(x_1(x_2 b), x_1(b x_2),(x_2 b)x_1, (b x_2)x_1, \\
			& & x_2(x_1 b), x_2(b x_1), (x_1 b)x_2, (bx_1)x_2);
		\end{array}$
    \item $\begin{array}{lll}
			(b_1 * b_2) \overline{*} x & = & W(b_1(b_2 x), b_1(x b_2),(b_2 x)b_1, (x b_2)b_1, \\
			& & b_2(b_1 x), b_2(x b_1),(b_1 x)b_2, (x b_1)b_2);
		\end{array}$
    \item $\begin{array}{lll}
			(b_1 * x) \overline{*} b_2 & = & W(b_1(x b_2), b_1(b_2 x), (x b_2)b_1, (b_2 x)b_1, \\
			& & x(b_1 b_2), x(b_2 b_1),(b_1 b_2)x, (b_2b_1)x);
		\end{array}$
\end{enumerate}
where $W$ indicates the same word in \textbf{Axiom 2}
corresponding to the choice of $*$ and $\overline{*}$.

Observe that axioms 1--4 above come from \textbf{Axiom 1}, while
axioms 5--8 come from \textbf{Axiom 2} by replacing each operation
with the corresponding action (notice that the group action
replaces the conjugation and not the group operation). These
axioms are nothing but the translation of the condition that one
obtains by considering the equivalence between actions and points
and expressing the action as the conjugation into the semidirect
product. More explicitly, given a split extension:
$$ \xymatrix{
    X \ar[r]^k & A \ar@<.5ex>[r]^p & B \ar@<.5ex>[l]^s
} $$ the corresponding action is given by:
\begin{itemize}
    \item[] $b\cdot x =  k^{-1}(s(b)+k(x)-s(b))$;
    \item[] $b*x = k^{-1}(s(b)*k(x))$.
\end{itemize}

A wider class of semi-abelian varieties is given by \emph{groups
with operations} introduced by Porter in \cite{Porter}. In that
class, the description of internal actions is similar to the one
given above; axioms 1--8 are replaced by suitable ones coming from
the identities of the corresponding algebraic theory.


\section{Commutators} \label{sec:comm}

Another classical property of characteristic subgroups of a group
is the fact that the commutator of two characteristic subgroups is
characteristic as well. In order to study this property in a
categorical setting, we will use an intrinsic definition of the
commutator of two subobjects. There are different possible
definitions. The first we consider is the so-called Huq commutator
\cite{Huq}. It can be constructed in the following way (see
\cite{Borceux-Bourn} and \cite{MaMe10-2}): given two subobjects $h
\colon H \rightarrowtail G$ and $k \colon K \rightarrowtail G$ of
an object $G$, the Huq commutator $[H, K]_G$ of $H$ and $K$ is
given by the following diagram:
\[ \xymatrix{ & H + K \ar[r]^{\Sigma_{H,K}} \ar[d]_{[h, k]} & H \times K \ar[d]
\\
[H, K]_G \ar@{ |>->}[r] & G \ar[r]_-{\pi} & \frac{G}{[H, K]_G}, }
\] where $\Sigma_{H,K}$ is the canonical map
\[ \Sigma_{H, K} = \langle [1,0], [0,1] \rangle = [\langle 1,0 \rangle, \langle\ 0,1 \rangle] \colon H+K \to H \times K \]
from the coproduct to the product and the commutative square is a
pushout. Then the Huq commutator appears as the kernel of the
morphism $\pi$. Being a kernel, the Huq commutator is always a
normal subobject, even if $H$ and $K$ are not.
\medskip

Another possible way to define the commutator is via the so-called
Higgins commutator \cite{MaMe10-2}. Given two subobjects $h \colon
H \rightarrowtail G$ and $k \colon K \rightarrowtail G$ of an
object $G$, let us denote by $\sigma_{H, K} \colon H \diamond K
\to H + K$ the kernel of the canonical morphism $\Sigma_{H,K}
\colon H + K \to H \times K$. The Higgins commutator $[H, K]$ of
$H$ and $K$ is the regular image of $H \diamond K$ under the
morphism $[h, k] \sigma_{H, K}$, as in the following diagram:
\[ \xymatrix{ H \diamond K \ar@{->>}[r] \ar@{ |>->}[d]_{\sigma_{H,
K}} & [H, K] \ar@{ >->}[d] \\
H + K \ar[r]_-{[h, k]} & G. } \] The Higgins commutator of $H$ and
$K$ is not necessarily a normal subobject of $G$, even when $H$
and $K$ are. In fact, its normalisation in $G$ is the Huq
commutator. A category $\C$ is said to satisfy the (NH) property
when the Higgins commutator of two normal subobjects is normal,
or, in other words, when Higgins and Huq commutators of normal
subobjects coincide. The \NH\ property is satisfied both by \LACC\
categories and by categories of interest (see \cite{CGrayVdL1}).

Let us observe that in the special case where $H=K=G$ and
$h=k=1_G$, $[G,G]$ is always normal in $G$, since the map $[1,1]
\colon G+G \to G$ is a regular epimorphism and in the semi-abelian
context regular images of normal subobjects along regular
epimorphisms are normal.
\medskip

Let us now start the study of the Huq commutator of two
characteristic subobjects.

\begin{prop} \label{prop:huq.char}
Let \C\ be a semi-abelian category satisfying the following
properties:
\begin{enumerate}
    \item the kernel functors $\Ker\colon \Pt_{B}(\C)\to \C$ preserve jointly strongly epimorphic pairs;
    \item the kernel functors $\Ker\colon \Pt_{B}(\C)\to \C$ preserve cokernels.
\end{enumerate}
If $H$ and $K$ are characteristic subobjects of $G$, then the Huq
commutator $[H,K]_G$ is a characteristic subobject of $G$.
\end{prop}

\begin{proof}
If $H$ and $K$ are characteristic subobjects of $G$, then, for
every action $\xi \colon B \flat G \to G$, there is a cospan in
$\Pt_B(\C)$:
$$ \xymatrix{
    H \ar@{ |>->}[r]^{k_1} \ar@{ >->}[d]_h & Y \ar@<.5ex>[r]^{p_1} \ar@{ >->}[d]
            & B \ar@<.5ex>[l]^{s_1} \ar@{=}[d] \\
    G \ar@{ |>->}[r]^-{i_G} & G \rtimes_\xi B \ar@<.5ex>[r]^-{p_B} & B \ar@<.5ex>[l]^-{i_B} \ar@{=}[d] \\
    K \ar@{ |>->}[r]^{k_2} \ar@{ >->}[u]^k & Z \ar@<.5ex>[r]^{p_2} \ar@{ >->}[u] & B \ar@<.5ex>[l]^{s_2}
} $$ The product $(Y,p_1,s_1)\times(Z,p_2,s_2)$ in $\Pt_B(\C)$ has
$H \times K$ as kernel. As already explained in the proof of the
Lemma \ref{lemma:jse<=>joins}, the kernel $N$ of the coproduct
$(Y,p_1,s_1)+(Z,p_2,s_2)$ is different, in general, from $H+K$;
however, under the assumption 1, the canonical map $u \colon H+K
\to N$ is a regular epimorphism. Now, consider the following
commutative diagram, where $\alpha$ is the arrow induced on
kernels by the canonical morphism $(Y,p_1,s_1)+(Z,p_2,s_2) \to
(Y,p_1,s_1)\times(Z,p_2,s_2)$ in $\Pt_B(\C)$, $\beta$ is induced
by $(Y,p_1,s_1)+(Z,p_2,s_2) \to (G \rtimes B,p_B,i_B)$, and
$j=\ker(\alpha)$:
$$ \xymatrix{
    H \diamond K \ar@{ |>->}[r]^{\sigma_{H,K}} \ar@{-|>}[d]_v
        & H + K \ar@{-|>}[r]^{\Sigma_{H,K}} \ar@{}[dr]|(.7){\ulcorner} \ar@{-|>}[d]^u
        & H \times K \ar@{=}[d] \\
    M \ar@{ |>->}[r]^j \ar[d] & N \ar@{-|>}[r]^-\alpha \ar@{}[dr]|(.7){\ulcorner} \ar[d]^\beta
        & H \times K \ar[d] \\
    [H,K]_G \ar@{ |>->}[r] & G \ar@{-|>}[r] & G/[H,K]_G
} $$ The arrow $v \colon H \diamond K \to M$ is a regular
epimorphism, thanks to the short five lemma. The Huq commutator
$[H,K]_G$ is defined as the kernel of the pushout of
$\Sigma_{H,K}$ along $\beta u=[h,k]$. Moreover,
$G/[H,K]_G=\Coker(\beta u \sigma_{H,K})$ by composition of
pushouts, and, as $v$ is a regular epimorhism, we also have
$G/[H,K]_G=\Coker(\beta j)$.

Remembering that kernel functors preserve kernels, $M$ is the
kernel of the object in $\Pt_B(\C)$ defined as the kernel of the
arrow $(Y,p_1,s_1)+(Z,p_2,s_2) \to (Y,p_1,s_1)\times(Z,p_2,s_2)$,
so $\beta j$ is the arrow induced on kernels by an arrow in
$\Pt_B(\C)$. Now, by hypothesis 2, the kernel functors preserve
cokernels, so that $G/[H,K]_G$ turns out to be the kernel of a
cokernel in $\Pt_B(\C)$. In particular, this means that there is
an action of $B$ on $G/[H,K]_G$ induced by the one on $G$. As a
consequence, we also have an action of $B$ on $[H,K]_G$, again
because the kernel functors preserve kernels.
\end{proof}

\begin{cor} \label{cor:der.char}
Let \C\ be a semi-abelian category satisfying the following
properties:
\begin{enumerate}
    \item the kernel functors $\Ker\colon \Pt_{B}(\C)\to \C$ preserve jointly strongly epimorphic pairs;
    \item the kernel functors $\Ker\colon \Pt_{B}(\C)\to \C$ preserve cokernels.
\end{enumerate}
The derived subobject $[G,G]$ is characteristic in $G$.
\end{cor}

\begin{cor}
Let \C\ be either a semi-abelian \LACC\ category or a
\emph{category of interest}. If $H$ and $K$ are characteristic
subobjects of $G$, then the Huq commutator $[H,K]_G$ is a
characteristic subobject of $G$.
\end{cor}

\begin{proof}
This depends on the fact that both classes of categories satisfy
the conditions of Proposition \ref{prop:huq.char}.

This is obvious in the case of \LACC\ categories. For
\emph{categories of interest}, it is proved in \cite{CGrayVdL2}.
\end{proof}

An analogous result can be stated for the Higgins commutator.

\begin{prop} \label{prop:hig.char}
Let \C\ be a semi-abelian category where the kernel functors
$\Ker\colon \Pt_{B}(\C)\to \C$ preserve jointly strongly
epimorphic pairs. If $H$ and $K$ are characteristic subobjects of
$G$, then the Higgins commutator $[H,K]$ is a characteristic
subobject of $G$.
\end{prop}

\begin{proof}
The result is a straightforward consequence of Proposition 6.2 in
\cite{CGrayVdL1}.
\end{proof}

As a consequence, we have:

\begin{cor}
Under the assumptions of the previous proposition, if $H$ and $K$
are characteristic subobjects of an object $X$ in \C\, then the
Huq commutator and the Higgins commutator of $H$ and $K$ coincide.
\end{cor}

In the category of (not necessarily unitary) rings, given a ring
$X$ and two subrings $H$ and $K$, the commutator $[H, K]$ is
nothing but the subring $HK$ of $X$ generated by $H$ and $K$.
Hence Proposition \ref{prop:hig.char} says that, if $H$ and $K$
are characteristic, $HK$ also is. The same happens in the category
of Lie algebras (over a commutative ring $R$), where the
commutator $[H, K]$ of two subalgebras is the Lie subalgebra
generated by $H$ and $K$.

\begin{remark}
When the category \C\ satisfies \NH, Propositions
\ref{prop:huq.char} and \ref{prop:hig.char} are both consequences
of Proposition $3.3$ in \cite{CGrayVdL1}, where it is shown that,
in the semi-abelian context, the property
$$ H,K \mbox{ characteristic in } X \quad\Rightarrow\quad [H,K] \mbox{ characteristic in } X $$
can be deduced directly from \NH.
\end{remark}

The fact that the Huq (or the Higgins) commutator of two
characteristic subobjects is characteristic is not true in a
general semi-abelian category. Not even the derived subobject of
an object (which is the same in the Higgins or in the Huq sense)
is characteristic in general, as the following example shows.

\begin{example} \label{example:[G,G]}
Let us consider the category \NARng\ of not necessarily
associative rings, i.e.\ abelian groups with a binary operation
which is distributive w.r.t.\ the group operation. Let $G$ be the
object in \NARng\ given by the free abelian group on two
generators $G = \Z x + \Z y$, endowed with a distributive binary
operation, defined on generators as:
$$ \begin{array}{r|cc}
    * & x & y \\
    \hline
    x & x & 0 \\
    y & 0 & 0 \\
\end{array} $$
Then the derived subobject $[G,G] = \Z x$ is an ideal (i.e.\ a
normal subobject) of $G$, but it is not characteristic in $G$.
Indeed, if we consider the object given by the abelian group \Z\
with trivial multiplication, $[G,G]$ is not stable under the
following action of \Z\ over $G$:
$$ \begin{array}{lcl}
  \Z \times G \rightarrow G\,, & & z * (\alpha x + \beta y) = (z\beta)x + (z\alpha)y\,, \\
  G \times \Z \rightarrow G\,, & & (\alpha x + \beta y) * z = (z\beta)x + (z\alpha)y\,.
\end{array} $$
We emphasize that $G$ is, in fact, an associative ring, but the
present is not a counterexample for the category \Rng\ of rings,
since the one described above is an action in \NARng\, but not in
\Rng. Indeed, according to the explicit description of actions
recalled at the end of Section \ref{sec:joins}, an action of \Z\
over $G$ in \NARng\ is just a pair of bilinear maps $\Z \times G
\rightarrow G$ and $G \times \Z \rightarrow G$, while an action in
\Rng\ must also satisfy some ``associativity'' axioms. In the
example above, the axiom
\[ z * (xx) = (z * x) x \]
is not satisfied, indeed $z*(xx)=z* x=zy$, while $(z*
x)x=(zy)x=0$.
\end{example}


\section{Centres and centralisers} \label{sec:centr}

Given a characteristic subgroup $H$ of a group $G$, its
centraliser $Z_G(H)$ is characteristic, too. In particular, the
centre of a group is always a characteristic subgroup. This is not
true in any semi-abelian category, as we will show later, so we
need to consider further hypotheses on the category in order to
get this property. In a semi-abelian category \C, given a
subobject $H$ of an object $G$, the centraliser of $H$ in $G$ is
the largest subobject $Z_G(H)$ of $G$ such that the Huq commutator
$[H, Z_G(H)]_G$ vanishes. The centre of an object $G$ is the
largest subobject $Z(G)$ of $G$ such that $[G, Z(G)] = 0$.
\medskip

The centres and centralisers do not always exist in a semi-abelian
category, and even when they exist, they can be difficult to
handle. Bourn and Janelidze introduced in \cite{BJ07} a
categorical context, namely \emph{action accessible categories},
in which the centres and the centralisers have an easy
description. We recall now the definition of action accessible
categories and their basic properties.
\medskip

Let $\C$ be a semi-abelian category. Fixed an object $K \in \C$, a
split extension with kernel $K$ is a diagram
$$ \xymatrix@=6ex{K \ar[r]^k  & A \ar@<2pt>[r]^p & B \ar@<2pt>[l]^s} \,, $$
such that $ps = 1_B$ and $k = \Ker(p)$. We denote such a split
extension by $(B, A, p, s, k)$. Given another split extension $(D,
C, q, t, l)$ with the same kernel $K$, a morphism of split
extensions
\begin{equation}
(g, f) \colon (B, A, p, s, k) \ \longrightarrow \ (D, C, q, t, l)
\end{equation}
is a pair $(g, f)$ of morphisms:
\begin{equation} \label{diag}
\begin{aligned}
\xymatrix@=6ex{
    K \ar[r]^k \ar@{=}[d] & A \ar[d]_f \ar@<2pt>[r]^p & B \ar[d]^g \ar@<2pt>[l]^s \\
    K \ar[r]^l & C \ar@<2pt>[r]^q & D \ar@<2pt>[l]^t
}
\end{aligned}
\end{equation}
such that $l = f k$, $q f = g p$ and $f s = t g$. Let us notice
that, since the  category $\C$ is protomodular, the pair $(k, s)$
is jointly (strongly) epimorphic, and then the morphism $f$ in
(\ref{diag}) is uniquely determined by $g$.

The split extensions with fixed kernel $K$ form a category,
denoted by $\mathsf{SplExt}_{\C}(K)$, or simply by
$\mathsf{SplExt}(K)$.

\begin{defi}[\cite{BJ07}]
\begin{itemize}
    \item []
    \item An object in $\mathsf{SplExt}(K)$ is said to be \emph{faithful} if any object
     in $\mathsf{SplExt}(K)$ admits at most one morphism into it.
    \item Split extensions with a morphism into a faithful one are called \emph{accessible}.
    \item If, for any $K \in \C$, every object in $\mathsf{SplExt}(K)$ is accessible, we say that the category $\C$ is \emph{action accessible}.
\end{itemize}
\end{defi}

In the case of groups, faithful extensions are those inducing a
group action of $B$ on $K$ (via conjugation in $A$) which is
faithful. Every split extension in $\Gp$ is accessible and a
morphism into a faithful one can be performed by taking the
quotients of $B$ and $A$ over the centraliser $Z_B(K)$, i.e.\ the
(normal) subobject of $A$ given by those elements of $B$ that
commute in $A$ with every element of $K$.
\medskip

In \cite{Montoli} it is proved that any category of interest in
the sense of \cite{Orzech} is action accessible. Examples of
action accessible categories are then groups, rings, associative
algebras, Lie algebras, Leibniz algebras and Poisson algebras, as
mentioned before.
\medskip

In the context of action accessible categories it is easy to
describe the centraliser of a normal subobject. We give now a
brief description of the construction, without proof (that can be
found, for example, in \cite{CiMa12}). Let $x \colon X \to A$ be a
normal subobject of $A$, and let $R[p]$ be the equivalence
relation on $A$ induced by $X$ (i.e. the kernel pair of the
quotient $p \colon A \to A/X$). Consider the following morphism of
split extensions, where the codomain is a faithful one (it exists
because the category is action accessible):
\[ \xymatrix{
    X \ar@{=}[d] \ar[r]^-{\langle x, 0 \rangle} & R[p] \ar@<2pt>[r]^-{r_0} \ar[d]_f
        & A \ar@<2pt>[l]^-{s_0} \ar[d]^g \\
    X \ar[r]_k & C \ar@<2pt>[r]^-q & D. \ar@<2pt>[l]^-t
} \] Then the kernel of $g$ is the centraliser $Z_A(X)$ of $X$ in
$A$. This implies, in particular, that in an action accessible
category the centraliser of a normal subobject is normal
\cite[Corollary 2.6]{CiMa12}, which is not always the case in
general semi-abelian categories, even when $Z_A(X)$ exists (see
examples in \cite{CiMa12}).
\medskip

We are now ready to prove that, in the context of action
accessible categories, the centraliser of a characteristic
subobject is characteristic.

\begin{lemma}[\cite{CGrayVdL1}] \label{lemma:lift.split}
Consider a split extension as in the bottom row of the diagram
$$ \xymatrix{
    K \ar@{ |>->}[d]_-{k} \ar@{-->}[r] & K' \ar@{-->}[d] \ar@{-->}@<.5ex>[r]
        & Z \ar@{==}[d] \ar@{-->}@<.5ex>[l] \\
    X \ar@{ |>->}[r]_-{x} & Y \ar@<.5ex>[r]^-{f} & Z \ar@<.5ex>[l]^-{s}
} $$ such that $x k$ is normal. Then this split extension lifts
along $k\colon K\to X$ to yield a normal monomorphism of split
extensions.
\end{lemma}
\begin{proof}
The needed lifting is obtained via the pullback of split
extensions in the diagram
\[
\xymatrix@!0@R=3.5em@C=5em{ & K \skewpullback \ar@{=}[dl] \ar@{{
|>}->}[ddd]|(.33){\hole}_(.6){k} \ar@{{ |>}-->}[rr] && K'
\skewpullback \ar[ld] \ar@{-->}[ddd] \ar@{-->}@<0.5ex>[rr] & & Z \skewpullback \ar@{==}[ddd] \ar[dl]_{s} \ar@{-->}@<0.5ex>[ll] \\
K \ar@{{ |>}->}[rr] \ar@{{ |>}->}[ddd]_{x k} & & R
\ar@<0.5ex>[rr]^(0.75){r_{1}}
\ar[ddd]_(.4){\langle r_{1},r_{2}\rangle } & & Y \ar@<0.5ex>[ll] \ar@{=}[ddd] &\\
& & & & &\\
& X \ar@{{ |>}->}[rr]_(.25){x}|-{\hole} \ar@{{ |>}->}[dl]^(.4){x}
& &
Y\ar@<0.5ex>[rr]^(.25){f}|-{\hole}\ar[dl]^(.3){\langle s f,1_{Y}\rangle}& & Z \ar@<0.5ex>[ll]^(.75){s}|-{\hole} \ar[dl]^(.4){s}\\
Y \ar@{{ |>}->}[rr]_{\langle 0,1_{Y}\rangle} & & Y\times Y
\ar@<0.5ex>[rr]^{\pi_1} & & Y \ar@<0.5ex>[ll]^{\langle
1_{Y},1_{Y}\rangle}}
\]
where $R$ is the denormalisation (\cite{Bourn2000, Borceux-Bourn})
of $x k$.
\end{proof}

\begin{lemma} \label{lemma:centr.intersection}
Let \C\ be a semi-abelian category where, for every normal
subobject $H \triangleleft G$, the centraliser $Z_G(H)$ of $H$ in
$G$ is normal in $G$. Then if $G'$ is a normal subobject of $G$,
$Z_{G'}(H)$ is also normal in $G$.
\end{lemma}

\begin{proof}
By definition of centraliser, $Z_{G'}(H)$ is the largest subobject of $G'$ such
that $[H,Z_{G'}(H)]_{G'}=0$. Hence, it is contained in both $G'$
and $Z_{G}(H)$, and it is the largest with this property, so it
is defined by the following pullback:
$$ \xymatrix{
    Z_{G'}(H) \ar@{ |>->}[r] \ar@{ |>->}[d] & Z_{G}(H) \ar@{ |>->}[d] \\
    G' \ar@{ |>->}[r] & G
} $$ In other words, $Z_{G'}(H) = Z_{G}(H) \wedge G'$ and it is
normal in $G$ as intersection of two normal subobjects.
\end{proof}

\begin{prop} \label{prop:centr.char}
Let \C\ be a semi-abelian category where, for every normal
subobject $H \triangleleft G$, the centraliser $Z_G(H)$ of $H$ in
$G$ is normal in $G$. Then if $H$ is a characteristic subobject of
$G$, $Z_G(H)$ is also characteristic in $G$.
\end{prop}

\begin{proof}
Consider an object $B$ and an action $\xi \colon B \flat G \to G$.
$G$ is a normal subobject of $G \rtimes_\xi B$; so, being
characteristic in $G$, $H$ is normal in $G \rtimes_\xi B$  by
Proposition \ref{prop:normality.lifts}. Hence, by Lemma
\ref{lemma:centr.intersection}, $Z_G(H)$ is a normal subobject of
$G \rtimes_\xi B$. Now, we can apply Lemma \ref{lemma:lift.split}
to the following situation:
$$ \xymatrix{
    & Z_G(H) \ar@{ |>->}[d] \ar@{ |>->}[dr] \\
    0 \ar[r] & G \ar@{ |>->}[r]_-{i_G} & G \rtimes_\xi B \ar@<.5ex>[r]^-{p_B} & B \ar@<.5ex>[l]^-{i_B} \ar[r] & 0
} $$ thus obtaining a morphism of split extensions:
$$ \xymatrix{
    0 \ar[r] & Z_G(H) \ar@{ |>->}[d] \ar@{ |>->}[r] & Z_{G \rtimes_{\xi} B}(G) \ar@{ |>->}[d] \ar@<.5ex>[r]
        & B \ar@{=}[d] \ar@<.5ex>[l] \ar[r] & 0 \\
    0 \ar[r] & G \ar@{ |>->}[r]_-{i_G} & G \rtimes_\xi B \ar@<.5ex>[r]^-{p_B} & B \ar@<.5ex>[l]^-{i_B} \ar[r] & 0
} $$ which gives the desired action $\xi' \colon B \flat Z_G(H)
\to Z_G(H)$ as a restriction of the action $\xi$.
\end{proof}

\begin{cor} \label{cor:centre.char}
Let \C\ be a semi-abelian category where, for every normal
subobject $H \triangleleft G$, the centraliser $Z_G(H)$ of $H$ in
$G$ is normal in $G$. Then the centre $Z(G)$ is a characteristic
subobject of $G$.
\end{cor}

In the category of (not necessarily unitary) rings, given an ideal
$H$ of a ring $G$, the centraliser $Z_G(H)$ is the annihilator of
$H$ in $G$, i.e.
\[ Z_G(H) = \{ g \in G \ | \ gh = hg = 0 \ \text{for all} \ h \in
H \}. \] Hence, if $H$ is characteristic in $G$, then the
annihilator of $H$ in $G$ is characteristic, as well. In
particular, for any ring $G$, the annihilator of $G$ is a
characteristic ideal of $G$. The same happens in the category of
Lie algebras over a commutative ring $R$. \medskip

Proposition \ref{prop:centr.char} and Corollary
\ref{cor:centre.char} are true, in particular, in semi-abelian
action accessible categories. However, they do not hold in any
semi-abelian category. The following is a counterexample.

\begin{example}
Let us consider again the category \NARng\ of not necessarily
associative rings and the object $G$ in \NARng\ described in
Example \ref{example:[G,G]}. The centre $Z(G) = \Z y$ is an ideal
(i.e.\ a normal subobject) of $G$, but it is not characteristic in
$G$, since it is not stable under the action of \Z\ over $G$
described in the same example.
\end{example}


\section{Induced morphisms between actors} \label{sec:actors}

In the category $\Gp$ of groups, if $H$ is a characteristic
subgroup of $G$, then there are induced morphisms $\Aut(G) \to
\Aut(H)$ and $\Aut(G) \to \Aut(G/H)$. This comes from the fact
that actions on $G$ (which are equivalent to split extensions with
kernel $G$, as already observed) are represented by the
automorphism group $\Aut(G)$, in the sense that an action of a
group $B$ on $G$ can be described simply as a group homomorphism
$B \to \Aut(G)$. We are going to show that the same induced
morphisms exist in a context in which internal actions (which are
equivalent to split extensions in a semi-abelian category) are
representable. Categories in which this happens are called
\emph{action representative} \cite{BJK2,Borceux-Bourn-SEC}. We now
recall the definition of an action representative category.

\begin{defi}[\cite{BJK2}]
A semi-abelian category \C\ is action representative if, for any
object $X \in \C$, there exists an object $\Act(X)$, called the
\emph{actor} of $X$, and a split extension
\[ \xymatrix{ X \ar[r] & X \rtimes \Act(X) \ar@<2pt>[r] & \Act(X) \ar@<2pt>[l] } \,, \]
called the \emph{split extension classifier} of $X$, such that,
for any split extension with kernel $X$:
\[ \xymatrix{ X \ar[r]^k & A \ar@<2pt>[r]^p & B \ar@<2pt>[l]^s } \]
there exists a unique morphism $\varphi \colon B \to \Act(X)$ such
that the following diagram commutes:
\[ \xymatrix{
    X \ar[r]^k \ar@{=}[d] & A \ar[d]_{\varphi_1} \ar@<2pt>[r]^p & B \ar@<2pt>[l]^s \ar[d]^{\varphi} \\
    X \ar[r] & X \rtimes \Act(X) \ar@<-2pt>[r] & \Act(X) \ar@<-2pt>[l]
} \] where the morphism $\varphi_1$ is uniquely determined by
$\varphi$ and the identity on $X$ (since $k$ and $s$ are jointly
strongly epimorphic).
\end{defi}

Examples of action representative categories are the category
$\Gp$ of groups, where the actor is the group of automorphisms,
and the category $\Lie$ of Lie algebras over a commutative ring
$R$, where the actor of a Lie algebra $X$ is the Lie algebra
$\Der(X)$ of derivations of $X$.
\medskip

It is well-known that the assignment $G \mapsto \Act(G)$ is not
functorial. Nevertheless, it behaves well with respect to
characteristic subobjects.

\begin{prop} \label{prop:mor.act}
Let \C\ be a semi-abelian action representative category. Every
characteristic subobject $h \colon H \rightarrowtail G$ induces a
morphism between split extension classifiers:
\begin{equation} \label{eq:map.sec}
\begin{aligned}
\xymatrix{
    G \ar@{ |>->}[r] \ar@{-|>}[d]_q & G \rtimes \Act(G) \ar@<.5ex>[r] \ar[d] & \Act(G) \ar@<.5ex>[l] \ar[d] \\
    G/H \ar@{ |>->}[r] & G/H \rtimes \Act(G/H) \ar@<.5ex>[r] & \Act(G/H) \ar@<.5ex>[l]
}
\end{aligned}
\end{equation}
and a morphism between actors: $\Act(G) \rightarrow \Act(H)$.
\end{prop}

\begin{proof}
As explained in Section \ref{sec:defi}, if $H$ is a characteristic
subobject of $G$, then, for every action $\xi \colon B \flat G \to
G$, there exists an exact sequence in $\Pt_B(\C)$:
$$ \xymatrix{
    H \ar@{ |>->}[r] \ar@{ |>->}[d]_h & Y \ar@<.5ex>[r] \ar@{ >->}[d]^\phi & B \ar@<.5ex>[l] \ar@{=}[d] \\
    G \ar@{ |>->}[r] \ar@{-|>}[d]_q & X \ar@<.5ex>[r] \ar@{-|>}[d]^\gamma & B \ar@<.5ex>[l] \ar@{=}[d] \\
    G/H \ar@{ |>->}[r] & Z \ar@<.5ex>[r] & B \ar@<.5ex>[l]
} $$ Since the category \C\ is action representative, we can
choose, in particular, $B=\Act(G)$ and the middle row to be the
split extension classifier of $G$. Thus, thanks to Proposition
\ref{prop:act.quot}, we have a morphism in $\Pt_{\Act(G)}(\C)$:
$$ \xymatrix{
    G \ar@{ |>->}[r] \ar@{-|>}[d] & G \rtimes \Act(G) \ar@<.5ex>[r] \ar[d] & \Act(G) \ar@<.5ex>[l] \ar@{=}[d] \\
    G/H \ar@{ |>->}[r] & Z \ar@<.5ex>[r] & \Act(G) \ar@<.5ex>[l]
} $$ By composing with the arrow to the split extension classifier
of $G/H$, we get the desired morphism (\ref{eq:map.sec}).

For the same reason, we also have a morphism:
$$ \xymatrix{
    H \ar@{ |>->}[r] \ar@{ |>->}[d]_h & Y \ar@<.5ex>[r] \ar@{ >->}[d] & \Act(G) \ar@<.5ex>[l] \ar@{=}[d] \\
    G \ar@{ |>->}[r] & G \rtimes \Act(G) \ar@<.5ex>[r] & \Act(G) \ar@<.5ex>[l]
} $$ The arrow from the upper split extension to the split
extension classifier of $H$ produces the morphism $\Act(G)
\rightarrow \Act(H)$.
\end{proof}

It is worth translating the above proposition in terms of internal
actions. The first assertion says that there exists a morphism
$\widetilde{q} \colon \Act(G) \to \Act(G/H)$ making the following
diagram commute:
$$ \xymatrix{
    \Act(G) \flat G \ar[r]^-{\widetilde{q}\flat q} \ar[d]_{\zeta_G} & \Act(G/H) \flat (G/H) \ar[d]^{\zeta_{G/H}} \\
    G \ar@{-|>}[r]_-q & G/H
} $$ where $\zeta_G$ and $\zeta_{G/H}$ are the canonical actions
of the actors. On the other hand, the second statement says that
there exists a morphism $\overline{h} \colon \Act(G) \to \Act(H)$
making this triangle commute:
$$ \xymatrix{
    \Act(G) \flat H \ar[r]^-{\overline{h}\flat 1} \ar[dr]_{\overline{\zeta_G}} & \Act(H) \flat H \ar[d]^{\zeta_H} \\
    & H
} $$ where $\overline{\zeta_G}$ is the action on $H$ induced by
$\zeta_G$ and $\zeta_H$ is the canonical action of the actor.
\medskip

Let us observe that any action representative category is action
accessible: indeed, it is easy to see that the split extension
classifier is a faithful split extension. On the other hand, the
category $\Rng$ of rings is action accessible \cite{BJ07} but not
action representative. In the case of action accessible
categories, one cannot recover the same properties described above
for action representative categories, because there can be many
faithful split extensions associated with a given one. However, as
observed in \cite{CiMa12}, there always exists a canonical
faithful split extension associated with a given one, and it has
properties analogous to the ones described above.
\medskip

Given a split extension
\[ \xymatrix{ X \ar[r]^k & A \ar@<2pt>[r]^p & B \ar@<2pt>[l]^s } \]
in a regular action accessible category, and a morphism of split
extensions with faithful codomain:
\[ \xymatrix{
    X \ar[r]^k \ar@{=}[d] & A \ar[d]_f \ar@<2pt>[r]^p & B \ar@<2pt>[l]^s \ar[d]^g \\
    X \ar[r] & C \ar@<2pt>[r]^q & D \ar@<2pt>[l]^t
} \] the canonical (regular epi, mono) factorization gives rise to
another faithful split extension:
\[ \xymatrix{
    X \ar[r]^k \ar@{=}[d] & A \ar@{->>}[d]_{e_f} \ar@<2pt>[r]^p & B \ar@<2pt>[l]^s \ar@{->>}[d]^{e_g} \\
    X \ar@{=}[d] \ar[r] & T_1 \ar@{ >->}[d]_{m_f} \ar@<2pt>[r] & T_0 \ar@{ >->}[d]^{m_g} \ar@<2pt>[l] \\
    X \ar[r] & C \ar@<2pt>[r]^q & D. \ar@<2pt>[l]^t
} \] The important fact here is that the faithful split extension
in the middle of the previous diagram does not depend on the
choice of the lower one, so it is a canonical faithful split
extension associated with $(A, B, p, s)$. The object $T_0$ is
actually the quotient $B/Z_B(X)$ of $B$ over the centraliser of
$X$ in $B$ (i.e. the largest subobject of $B$ commuting with $X$
in $A$), while $T_1$ is the quotient $A/Z_B(X)$.
\medskip

As above, let $H$ be a characteristic subobject of $G$. Then, for
every action $\xi \colon B \flat G \to G$, there exists an exact
sequence in $\Pt_B(\C)$ as in diagram (\ref{eq:cok.pt}). Let
$$ \xymatrix{
    G \ar@{ |>->}[r] \ar@{=}[d] & X \ar@<.5ex>[r] \ar@{-|>}[d] & B \ar@<.5ex>[l] \ar@{-|>}[d] \\
    G \ar@{ |>->}[r] & T_1(B,G,\xi) \ar@<.5ex>[r] & T_0(B,G,\xi) \ar@<.5ex>[l]
} $$ be the morphism onto the canonical faithful split extension
(and similarly for the induced split extensions of kernels $H$ and
$G/H$).

\begin{prop}
Let \C\ be a semi-abelian action accessible category. Every
characterstic subobject $h \colon H \rightarrowtail G$ induces a
morphism between canonical faithful split extensions:
\begin{equation} \label{eq:map.cf}
\begin{aligned}
\xymatrix{
    G \ar@{ |>->}[r] \ar@{-|>}[d]_q & T_1(B,G,\xi) \ar@<.5ex>[r] \ar[d] & T_0(B,G,\xi) \ar@<.5ex>[l] \ar[d] \\
    G/H \ar@{ |>->}[r] & T_1(B,G/H,\widetilde{\xi}) \ar@<.5ex>[r] & T_0(B,G/H,\widetilde{\xi}) \ar@<.5ex>[l]
}
\end{aligned}
\end{equation}
and a morphism: $T_0(B,G,\xi) \rightarrow
T_0(B,H,\overline{\xi})$.
\end{prop}

\begin{proof}
As explained above, the object $T_0(B,G,\xi)$ is nothing but the
quotient $B/Z_G(B)$, and $T_1(B,G,\xi) \cong X/Z_G(B)$, and
similarly for $T_i(B,H,\overline{\xi})$ and
$T_i(B,G/H,\widetilde{\xi})$. The desired morphism
(\ref{eq:map.cf}) will be the bottom rectangle in the following
commutative diagram:
$$ \xymatrix@!0@R=3.5em@C=5em{
    & G \ar@{-|>}[dl]_q \ar@{=}[dd]|(.5){\hole} \ar@{{ |>}->}[rr]
        & & X \ar@{-|>}[ld]^(.4){\gamma} \ar@{-|>}[dd]|(.47){\hole}|(.52){\hole} \ar@<0.5ex>[rr]
        & & B \ar@{-|>}[dd] \ar@{=}[dl] \ar@<0.5ex>[ll] \\
    G/H \ar@{{ |>}->}[rr] \ar@{=}[dd] & & Z \ar@<0.5ex>[rr] \ar@{-|>}[dd] & & B \ar@<0.5ex>[ll] \ar@{-|>}[dd] \\
    & G \ar@{{ |>}->}[rr]|(.5){\hole} \ar@{-|>}[dl]^(.4){q}
        & & T_1(B,G,\xi) \ar@<0.5ex>[rr]|-{\hole} \ar@{--|>}[dl]^(.4){q_1}
        & & T_0(B,G,\xi) \ar@<0.5ex>[ll]|-{\hole} \ar@{--|>}[dl]^(.4){q_0} \\
    G/H \ar@{{ |>}->}[rr] & & T_1(B,G/H,\widetilde{\xi}) \ar@<0.5ex>[rr]
        & & T_0(B,G/H,\widetilde{\xi}) \ar@<0.5ex>[ll]
} $$ It is constructed as follows. By definition, the centraliser
$Z_G(B)$ is such that $[G,Z_G(B)]_X=0$. Composing with $\gamma$,
we also have $[G/H,Z_G(B)]_Z=0$, so that $Z_G(B) \leq Z_{G/H}(B)$,
and this induces the arrow $q_0$ between the corresponding
cokernels over $B$. On the other hand, $q_1$ is the arrow which
completes the following morphism of short exact sequences:
$$ \xymatrix{
    Z_G(B) \ar@{ |>->}[r] \ar@{ >->}[d] & X \ar@{-|>}[r] \ar@{-|>}[d]^{\gamma} & T_1(B,G,\xi) \ar@{--|>}[d]^{q_1} \\
    Z_{G/H}(B) \ar@{ |>->}[r] & Z \ar@{-|>}[r] & T_1(B,G/H,\widetilde{\xi})
} $$ To prove the second assertion, consider the morphism below in
$\Pt_B(\C)$:
$$ \xymatrix{
    H \ar@{ |>->}[r] \ar@{ |>->}[d]_h & Y \ar@<.5ex>[r] \ar@{ >->}[d]^\phi & B \ar@<.5ex>[l] \ar@{=}[d] \\
    G \ar@{ |>->}[r] & X \ar@<.5ex>[r] & B \ar@<.5ex>[l]
} $$ By definition, $[G,Z_G(B)]_X=0$ and, as a consequence,
$[H,Z_G(B)]_X=0$. Since $\phi$ is monomorphic, this implies that
$[H,Z_G(B)]_Y=0$ too, hence $Z_G(B) \leq Z_H(B)$. The morphism
$T_0(B,G,\xi) \rightarrow T_0(B,H,\overline{\xi})$ is the one
induced on the corresponding cokernels over $B$.
\end{proof}


\section{Summarising table}

We conclude this paper by displaying, in the following table, a
list of the properties of characteristic subobjects we proved. In
the second column, a categorical context is indicated for each
property to hold. In many cases, it is not the most general one;
possible extensions to wider contexts are suggested by the proofs.

{\small
\begin{center}
\begin{tabular}{ccc}
    \hline 
    \sffamily Property & \sffamily True in & \sffamily Reference \\ 
    \hline \hline
    $H \car G \quad \Rightarrow \quad H \triangleleft G$ & \C\ semi-abelian & \ref{cor:char=>norm} \\
    \hline
    $H \car K \triangleleft G \quad \Rightarrow \quad H \triangleleft G$
        & \C\ semi-abelian & \ref{prop:normality.lifts} \\
    \hline
    $H \car K \car G \quad \Rightarrow \quad H \car G$ & \C\ semi-abelian & \ref{prop:char.lifts} \\
    \hline
    $H_i \car G \quad (i \in I) \quad \Rightarrow \quad \bigwedge_{i \in I} H_i \car G$
        & \C\ semi-abelian & \ref{prop:meet.char} \\
    \hline
    $H \car G,\: B \mbox{ acts on } G \quad \Rightarrow \quad B \mbox{ acts on } G/H$
        & \C\ semi-abelian & \ref{prop:act.quot} \\
    \hline
    $\begin{array}{c}
        H \leq K \leq G,\quad H \car G,\quad K/H \car G/H \\
        \Rightarrow \quad K \car G
    \end{array}$
        & \C\ semi-abelian & \ref{prop:K.char} \\
    \hline
    $\begin{array}{c}
        H \car G,\quad (R \mbox{ kernel pair of } G \to G/H) \\
        \Rightarrow \quad R \mbox{ closed under actions on } G
    \end{array}$
        & \C\ semi-abelian & \ref{prop:R.char} \\
    \hline
    $H,K \car G \quad \Rightarrow \quad H \vee K \car G$
        &   \begin{tabular}{c}
            \C\ \LACC\ \\
            \C\ category of interest
        \end{tabular}
        & \ref{prop:join.char} \\
    \hline
    $[G,G] \car G$
        &   \begin{tabular}{c}
            \C\ \LACC\ \\
            \C\ category of interest
        \end{tabular}
        & \ref{cor:der.char} \\
    \hline
    $H,K \car G \quad \Rightarrow \quad [H,K] \car G$
        &   \begin{tabular}{c}
            \C\ \LACC\ \\
            \C\ category of interest
        \end{tabular}
        & \ref{prop:huq.char} \\
    \hline
    $Z(G) \car G$ & \C\ action accessible & \ref{cor:centre.char} \\
    \hline
    $H \car G \quad \Rightarrow \quad Z_G(H) \car G$ & \C\ action accessible & \ref{prop:centr.char} \\
    \hline
    $H \car G \quad \Rightarrow \quad \left\{ \begin{array}{l} \Act(G) \to \Act(G/H) \\ \Act(G) \to \Act(H) \end{array} \right.$
        & \C\ action representative & \ref{prop:mor.act} \\
    \hline
\end{tabular}
\end{center}
}

\end{document}